\documentclass{aims}
\usepackage{amsmath}\usepackage{amssymb}
  \usepackage{paralist}
  \usepackage{graphics}
  \usepackage{epsfig}
\usepackage{graphicx}  \usepackage{epstopdf}
 \usepackage[colorlinks=true]{hyperref}
\hypersetup{urlcolor=blue, citecolor=red}

\def\currentvolume{X}
\def\currentissue{X}
\def\currentyear{200X}
\def\currentmonth{XX}
\def\ppages{X--XX}
\def\DOI{10.3934/xx.xx.xx.xx}

\newtheorem{theorem}{Theorem}[section]

\theoremstyle{definition}

\title[Impact of delay on HIV-1 dynamics]
      {Impact of delay on HIV-1 dynamics of f{}ighting a virus with
another virus}

\author[Yun Tian, Yu Bai and Pei Yu]{}



 \email{ytian56@uwo.ca}
 \email{youyou1123@live.ca}
 \email{pyu@uwo.ca}

\begin{document}
\maketitle

\centerline{\scshape Yun Tian, Yu Bai, Pei Yu}
\medskip
{\footnotesize
 \centerline{Department of Applied Mathematics,}
   \centerline{ Western University, London, Ontario N6A 5B7, Canada}
} 



\bigskip


\centerline{(Communicated by the associate editor name)}


\begin{abstract}

In this paper, we propose a mathematical model for HIV-1 infection
with intracellular delay. The model examines a viral-therapy for controlling
infections through recombining HIV-1 virus with a genetically modified
virus. For this model, the basic reproduction number $\mathcal{R}_0$
are identified and its threshold properties are discussed.
When $\mathcal{R}_0 < 1$, the infection-free equilibrium $E_0$ is globally
asymptotically stable. When $\mathcal{R}_0 > 1$, $E_0$ becomes
unstable and there occurs the single-infection equilibrium $E_s$,
and $E_0$ and $E_s$ exchange their stability at the transcritical
point $\mathcal{R}_0 =1$.
If $1< \mathcal{R}_0 < R_1$, where $R_1$ is a positive constant explicitly
depending on the model parameters, $E_s$ is globally asymptotically stable,
while when $\mathcal{R}_0 > R_1$, $E_s$ loses its
stability to the double-infection equilibrium $E_d$.
There exist a constant $R_2$ such that $E_d$ is asymptotically
stable if $R_1<\mathcal R_0 < R_2$, and $E_s$ and $E_d$ exchange their
stability at the transcritical point $\mathcal{R}_0 =R_1$.
We use one numerical example
to determine the largest range of $\mathcal R_0$ for the local
stability of $E_d$ and existence of Hopf bifurcation. Some simulations
are performed to support the theoretical results.
These results show that the delay plays an
important role in determining the dynamic behaviour of the system.
In the normal range of values, the delay may change the dynamic behaviour
quantitatively, such as greatly reducing the amplitudes of oscillations,
or even qualitatively changes the dynamical behaviour such as revoking
oscillating solutions to equilibrium solutions.
This suggests that the delay is a very important
fact which should not be missed in HIV-1 modelling.
\end{abstract}


\section{Introduction}

Human immunodeficiency virus (HIV) is a serious mortal lentivirus,
which can cause acquired immunodeficiency syndrome (AIDS).
Reports have known that many people are killed by AIDS every year,
and yet, until today, there is no effective way to cure
the AIDS. Thus, many scientists and researchers have been focusing on
the study of controlling the infections.
One of the approaches developed recently, offered by genetic
engineering, is to use recombinant virus capable
of controlling infections of HIV \cite{Wagner1999,Nolan1997}.
Recently, Revilla and Garcia-Ramos established a 5-dimensional
ordinary differential system to investigate the control of the infections
by introducing a recombinant virus to fight the virus~\cite{Revilla2003}.
Later, this model was studied by Jiang {\it et al.}~\cite{Jiang2009}
in detail to show various bifurcation patters and rich dynamics,
as well as a control study given in~\cite{YZ2012}
by introducing a constant injection rate of the recombinant virus to
this model.

A standard and classic in-host model for HIV infection
can be described by the following differential equations:
\begin{equation}\label{a1}
\begin{array}{lll}
\dot x&\!\!\!=\!\!\!& \lambda-dx-\beta xv,\\[0mm]
\dot y&\!\!\!=\!\!\!& \beta xv-ay,\\[0mm]
\dot v&\!\!\!=\!\!\!& ky-pv,
\end{array}
\end{equation}
where $x(t)$, $y(t)$, $v(t)$ are the density of virus-free host cells,
infected cells, and a pathogen virus, respectively, at time $t$.
The production rate and death rate for the healthy cells are respectively
$\lambda$ and $d$. $\beta$ is the constant rate at which a T-cell
is contacted by the virus. It is also assumed that once cells are infected,
they may die at a rate $a$ due to the action of either the virus or the
immune system, and each produces the pathogens at a rate $k$ during
their life which on average has length $1/a$.

In~\cite{Revilla2003}, a second virus is added into
model (\ref{a1}) which may cause the infected cells to have a second
infection, called double-infection, leading to a modified model as
\begin{equation}\label{b2}
\begin{array}{cll}
\dot x&\!\!\!=\!\!\!& \lambda-dx-\beta xv,\\[0mm]
\dot y&\!\!\!=\!\!\!& \beta xv-ay-\alpha wy,\\[0mm]
\dot z&\!\!\!=\!\!\!& \alpha wy-bz,\\[0mm]
\dot v&\!\!\!=\!\!\!& ky-pv,\\[0mm]
\dot w&\!\!\!=\!\!\!& cz-qw,
\end{array}
\end{equation}
where $w(t)$ and $z(t)$ are the recombinant (genetically modified) virus
and double-infected cells. After the second virus is enrolled,
once the cells which have been infected by the pathogens are infected
again by the recombinant, they can be turned into
double-infected cells at a rate $\alpha\omega y$, where the recombinants are
removed at a rate $qw$. The double infected cells die at a rate $bz$, and
release recombinants at rate $cz$. Having established the model (\ref{b2}),
the authors of~\cite{Revilla2003} analyzed the structure of equilibrium
solutions and presented some simulations. Later, in~\cite{Jiang2009}, the
authors fully analyzed the stability of all three equilibrium solutions
and bifurcations between these equilibria, as well as proved the existence of
Hopf bifurcation. Further, in~\cite{YZ2012}, the fifth equation of
model (\ref{b2}) is modified as $\dot w= \eta + cz-qw $, where $\eta$ is
a control parameter to measure the injection rate of the recombinant,
and then a complete dynamical analysis is given in this article, showing
that increasing $\eta$ is beneficial for controlling/eliminating the
HIV virus~\cite{YZ2012}.

In this paper, to further improve the model ({\ref{b2}),
we introduce a time lag into the model (\ref{b2}),
since in real situation, time is needed for the virus to contact
a target cell and then the contacted cells to become actively affected.
This can be described by the eclipse phase of the virus life cycle.
Moreover, we assume that the probability density that a cell still
remains infected for $\tau$ time units after being contacted
by the virus obeys an exponentially decay function.
Therefore, following the line of \cite{Zhu2008,Zhu2009},
model (\ref{b2}) can be modified to
\begin{equation}\label{b3}
\begin{array}{cll}
\dot x(t)&\!\!\!=\!\!\!& \lambda-dx(t)-\beta x(t)v(t),\\[0.5mm]
\dot y(t)&\!\!\!=\!\!\!& \beta e^{-a\tau} x(t-\tau)v(t-\tau)-ay(t)
-\alpha w(t)y(t),\\[0.5mm]
\dot z(t)&\!\!\!=\!\!\!& \alpha w(t)y(t)-bz(t),\\[0.5mm]
\dot v(t)&\!\!\!=\!\!\!& ky(t)-pv(t),\\[0.5mm]
\dot w(t)&\!\!\!=\!\!\!& cz(t)-qw(t),
\end{array}
\end{equation}
where $\tau$ denotes the average time for a viral particle to go through
the eclipse phase. Because the dimension of the system is higher than two,
model (\ref{b3}) may exhibit some interesting dynamic behaviors (Hopf
bifurcation, limit cycles and even chaos), which would make the analysis of
the system more complicated. Thus, the main goal of this paper focuses on
dynamical behaviour of the system with delay, in particular, on
equilibrium solutions and their bifurcations. More importantly,
we want to find the impact of the delay on the dynamical properties.

The rest of this paper is organized as follows.
In next section, for system (\ref{b3}) we will discuss the well-posedness of
the solutions, equilibria and their stability.
Also, in order to properly define biologically meaningful equilibria,
the basic reproduction number $\mathcal{R}_0$ will be defined.
In Sections 3, 4 and 5, we analyze the stability of the three equilibria:
disease-free equilibrium $E_0$ , single-infection equilibrium $E_s$,
and double-infection equilibrium $E_d$. It will be shown that
$E_0$ is globally asymptotically stable for $0<\mathcal{R}_0<1$,
$E_s$ is globally asymptotically stable for $1<\mathcal{R}_0<R_1$,
where $R_1>1$ is a constant defined in terms of the system parameters,
and $E_d$ is asymptotically stable for $R_1<\mathcal{R}_0<R_h$,
where $R_h$ denotes a Hopf critical point from which a family of
limit cycles bifurcate.
A numerical example is present in Section 6 to demonstrate the
theoretical predictions. Finally, conclusion and discussion are drawn
in Section 7.

\section{Well-posedness, boundedness of solutions, equilibria and basic
reproduction number}

Because of biological reasons, all variables in model (\ref{b3})
must be non-negative. Therefore, for any non-negative initial values,
the corresponding solution must remain non-negative. We have the following
result.

\begin{theorem} \label{result0}
All solutions of system (\ref{b3}) remain non-negative, provided the given
conditions are non-negative, and bounded.
\end{theorem}

\begin{proof}
For convenience, let $X=C([-\tau,0];R^5)$ be the Banach space of
continuous mapping from $[-\tau,0]$ to $R^5$ equipped with the sup-norm.
Let $\mathbf{x}(t)=(x(t),y(t),z(t),\\
v(t),w(t))^T$
and $\mathbf{x}_t(\theta)=\mathbf{x}(t+\theta)$ for $\theta\in[-\tau,0]$.
By the fundamental theory of FDEs (see, e.g. \cite{Hale1993}),
for any initial condition $\phi\in X$ with $\phi\ge0$,
we know that there exists a unique solution $\mathbf x(t,\phi)$
satisfying $\mathbf x(\theta,\phi)=\phi(\theta)$, $\theta\in[-\tau,0]$.

System (\ref{b3}) can be written as $\dot{\mathbf{x}}(t)=\mathbf{f}(\mathbf{x}_t)$,
where
\begin{equation*}
\mathbf f(\mathbf{x}_t)=
\left(\begin{array}{c}
\lambda-dx_t(0)-\beta x_t(0)v_t(0)\\
\beta e^{-a\tau}x_t(-\tau)v_t(-\tau)-ay_t(0)-\alpha w_t(0)y_t(0)\\
\alpha w_t(0)y_t(0)-bz_t(0)\\
ky_t(0)-pv_t(0)\\
cz_t(0)-qw_t(0)
\end{array}\right).
\end{equation*}
It is easy to see that if any $\phi\in X$ satisfies $\phi\ge0$, $\phi_i(0)=0$ for some $i$,
then $\mathbf f_i(\phi)\ge0$.
Therefore, according to
Theorem 2.1 (on page 81) in~\cite{Smith1995} we know that $\mathbf x(t,\phi)\ge0$
for all $t\ge0$ in its maximal interval of existence if $\phi\ge0$.



Next, to show the boundedness of the solution $(x(t),y(t),z(t),v(t),w(t))$,
we define
\begin{equation*}
B(t)= cke^{-a\tau}x(t)+cky(t+\tau)+ckz(t+\tau)
+\frac{ac}{2} v(t+\tau)+\frac{bk}{2} w(t+\tau).
\end{equation*}
Then, the derivative of $B(t)$ with respective to time $t$ along the
solution of trajectory of system (\ref{b3}) is given by
\begin{equation*}
\begin{split}
\frac{dB(t)}{dt}\Big|_{(\ref{b3})}
=&\ cke^{-a\tau}\big[\lambda-dx(t)-\beta v(t)x(t)\big]\\[-1.0ex]
 & +ck\big[\beta e^{-a\tau}v(t)x(t)-ay(t+\tau)-\alpha w(t+\tau)y(t+\tau)\big]\\
  &+ck\big[\alpha w(t+\tau)y(t+\tau)-bz(t+\tau)\big]\\[-0.5ex]
  &+\frac{ac}{2}\big[ky(t+\tau)-pv(t+\tau)\big]
+\frac{bk}{2}\big[cz(t+\tau)-qw(t+\tau)\big]\\[-1.0ex]
 =& \ cke^{-a\tau}\lambda-dcke^{-a\tau}x(t)
-\frac{a}{2}cky(t+\tau)-\frac{b}{2}ckz(t+\tau)\\[-1.0ex]
  & -p\frac{ac}{2}v(t+\tau)-q\frac{bk}{2}w(t+\tau)\\
\leq & \ cke^{-a\tau}\lambda-mB(t) \left\{\!\! \begin{array}{ll}
< 0 & {\rm for} \ \ B(t) > \displaystyle\frac{ck}{m} e^{- a \tau}, \\[1.5ex]
> 0 & {\rm for} \ \ B(t) < \displaystyle\frac{ck}{m} e^{- a \tau},
\end{array}  \right.
\end{split}
\end{equation*}
where $ m = \min\{ d, \frac{a}{2}, \frac{b}{2}, p, q \}$
This implies that $B(t)$ is bounded, so are $x(t)$, $y(t)$, $z(t)$, $v(t)$
and $w(t)$.
\end{proof}

Model (\ref{b3}) has three possible biologically meaningful equilibria:
disease-free equilibrium $E_0$, single-infection equilibrium $E_s$
and double-infection equilibrium $E_d$, given below:
\begin{equation*}
\begin{split}
E_0 =& \Big(\frac{\lambda}{d},\,0,\,0,\,0,\,0\Big),\\
E_s =& \Big(\frac{ap}{\beta ke^{-a\tau}},\
\frac{k\beta\lambda e^{-a\tau}-adp}{\beta ak},\ 0,\
\frac{k\beta\lambda e^{-a\tau}-adp}{\beta ap},\ 0\Big),\\
E_d =& \Big(\frac{\lambda \alpha cp}{d\alpha cp+\beta bkq},\
\frac{bq}{\alpha c},\
\frac{q(\alpha\beta\lambda cke^{-a\tau}-\beta abkq-\alpha acdp)}
{\alpha c(\beta bkq+\alpha cdp)},\\
&\quad \frac{bkq}{\alpha cp}, \
\frac{\alpha\beta\lambda cke^{-a\tau}-\beta abkq-\alpha acdp}
{\alpha (\beta bkq+\alpha cdp)}\Big).
\end{split}
\end{equation*}

We define
\begin{equation*}
\mathcal{R}_0 \triangleq \frac{\lambda}{d}\cdot\frac{\beta e^{-a\tau}}{a}\cdot \frac k  p= \frac{k\beta\lambda}{adp}e^{-a\tau},
\end{equation*}
where $\frac{\lambda}{d}$ is the average number of healthy cells available for infection,
$\frac{\beta e^{-a\tau}}{a}$ is the average number of host cells that each HIV virus infects,
and $\frac k p$ is the average number of HIV viruses that an infected cell produces.
Therefore, $\mathcal R_0$ is the basic reproduction number.

It is seen that the disease-free equilibrium is independent of the
delay. If $\mathcal{R}_0<1$, $E_0$ is the only biologically meaningful equilibrium.
If $\mathcal{R}_0>1$, there is another biologically meaningful equilibrium
$E_s$ (single-infection equilibrium). The double-infection equilibrium
$E_d$ exists (biologically meaningful) if and only if $R_d>1$, where
\begin{equation*}
R_d = \frac{\alpha\beta\lambda cke^{-a\tau}-\alpha acdp}{\beta abkq}
=\frac{\alpha cdp}{\beta bkq}(\mathcal{R}_0-1).
\end{equation*}
Hence,
\begin{equation*}
R_d>1 \Leftrightarrow \mathcal{R}_0>R_1,\ \
\mbox{where}\ R_1=1+\frac{\beta bkq}{\alpha cdp}.
\end{equation*}
Note that $R_1$ is independent of the delay.

\section{Stability of the disease-free equilibrium  $E_0$}
First, for the local stability of $E_0$, we have the following theorem.

\begin{theorem} \label{Thm1}
When $\mathcal{R}_0<1$, the disease-free equilibrium $E_0$ is locally
asymptotically stable; when $\mathcal{R}_0>1$, $E_0$ becomes unstable
and the single-infection equilibrium $E_s$ occurs.
\end{theorem}

\begin{proof}
The linearized system of (\ref{b3}) at the disease-free equilibrium $E_0$ is
\begin{equation*}
\begin{array}{cll}
\dot x(t)&\!\!\!=\!\!\!& -dx(t)-\frac{\beta\lambda}{d} v(t),\\[1mm]
\dot y(t)&\!\!\!=\!\!\!& \beta e^{-a\tau}\frac{\lambda}{d} v(t-\tau)-ay(t),
\\[1mm]
\dot z(t)&\!\!\!=\!\!\!& -bz(t),\\[0.5mm]
\dot v(t)&\!\!\!=\!\!\!& ky(t)-pv(t),\\[0mm]
\dot w(t)&\!\!\!=\!\!\!& cz(t)-qw(t),
\end{array}
\end{equation*}
for which the characteristic equation is given by
$$
(\xi+d)(\xi+b)(\xi+q) \Big[\xi^2+(a+p)\xi+ap
-\frac{\beta\lambda k}{d}e^{-(a+\xi)\tau}\Big]=0.
$$

Obviously, for the local stability of $E_0$, it suf{}f{}ices to only
consider the following equation
\begin{equation}\label{b6}
D_0(\xi)=\xi^2+(a+p)\xi+ap-\frac{\beta\lambda k}{d}e^{-(a+\xi)\tau}=0.
\end{equation}
If $\mathcal{R}_0>1$, it is easy to show for real $\xi$ that
$$
D_0(0)=ap(1-\mathcal{R}_0)<0, \quad \lim_{\xi\rightarrow+\infty}
D_0({\xi})=+\infty.
$$
Hence, $D_0(\xi)=0$ has at least one positive real root. Therefore,
if $\mathcal{R}_0>1$, the infection-free equilibrium $E_0$ is unstable.

Next, consider $\mathcal{R}_0<1$. When $\tau=0$, equation (\ref{b6}) becomes
\begin{equation}\label{b7}
\xi^2+(a+p)\xi+ap- \frac{\beta\lambda k}{d}=0.
\end{equation}
In order for the two roots of (\ref{b7}) to have negative real part,
it requires $ap-\beta\lambda k/d>0$, which is equivalent to
$\mathcal{R}_0|_{\tau=0}<1$. Thus, all the roots of (\ref{b7})
have negative real part when $\mathcal{R}_0<1$. From \cite{Busenberg1993},
we know that all the roots of (\ref{b6}) continuously depend on $\tau$.
And the assumption
\begin{equation}\label{b8}
\limsup\left\{ \left| \frac{Q(\xi,\tau)}{P(\xi,\tau)} \right|:
|\xi|\rightarrow\infty,\ {\rm Re}(\xi)\geq0 \right\}<1,
\ \ {\rm for}\ {\rm any} \ \tau,
\end{equation}
could ensure that there are no roots existing in the infinity for equations
in the form $P(\xi,\tau)+Q(\xi,\tau)e^{-\xi\tau}=0$ (see \cite{Beretta2002}).
Obviously, (\ref{b8}) holds here for (\ref{b6}), and hence
${\rm Re}(\xi) < +\infty$ for any root $\xi$ of (\ref{b6})
when $\mathcal{R}_0<1$.
As a result, for $\mathcal{R}_0<1$, the only possibility for the roots
of equation (\ref{b6}) to enter into the right half plane is to cross
the imaginary axis when $\tau$ increases. Thus, we define $\xi=i\varpi$,
$(\varpi>0)$, to be a purely imaginary root of (\ref{b6}).
Then we get
\begin{equation}\label{b9}
-\varpi^2+i(a+p)\varpi+ap- \frac{k\beta\lambda}{d} e^{-(a+i\varpi)\tau}=0.
\end{equation}
Taking moduli of (\ref{b9}) gives
\begin{equation*}
H_0(\varpi^2)=\varpi^4+(a^2+p^2)\varpi^2+a^2p^2
-\Big( \frac{k\beta\lambda}{d} e^{-a\tau} \Big)^2=0.
\end{equation*}
Clearly, $H_0(\varpi^2)$ has no positive real roots if $\mathcal{R}_0<1$.
Therefore, all the roots of (\ref{b6}) have negative real part
if $\mathcal{R}_0<1$.
\end{proof}

Further, for the global stability of $E_0$, we have the following result.

\begin{theorem} \label{Thm2}
If $\mathcal{R}_0<1$, the disease-free equilibrium $E_0$ is globally
asymptotically stable, implying that none of the two virus can invade
regardless of the initial load.
\end{theorem}

\begin{proof}
We construct the following Lyapunov function:
\begin{equation*}
\begin{split}
V_0 =& \ \frac{e^{-a\tau}}{2} \Big[ x(t)-\frac{\lambda}{d} \Big]^2
+\frac{\lambda}{d}y(t) +\frac{\lambda}{d}z(t)+\frac{a\lambda}{dk}v(t)
+\frac{b\lambda}{cd}w(t) \\
& +\frac{\lambda\beta}{d} e^{-a\tau}
\int_{t-\tau}^t{x(\eta)v(\eta)}d\eta.
\end{split}
\end{equation*}
Using non-negativity of the solution and $\mathcal{R}_0<1$, the derivative
of $V_0$ with respective to time $t$ along the solution of system (\ref{b3})
can be expressed as
\begin{equation*}
\begin{split}
\frac{dV_0}{dt}\Big|_{(\ref{b3})} =& \ e^{-a\tau} \Big[x(t)
-\frac{\lambda}{d} \Big] \big[\lambda-dx(t)-\beta v(t)x(t)\big]\\
 &+\frac{\lambda}{d}\big[\beta e^{-a\tau}x(t-\tau)v(t-\tau)-ay(t)-bz(t)\big]\\
 &+\frac{a\lambda}{dk}\big[ky(t)-pv(t)\big]
+\frac{b\lambda}{cd}\big[cz(t)-qw(t)\big]\\
 &+\frac{\lambda\beta}{d} e^{-a\tau}\big[x(t)v(t)-x(t-\tau)v(t-\tau)\big]\\
=&-e^{-a\tau}\Big[x(t)-\frac{\lambda}{d}\Big]^2
\big[d+\beta v(t)\big]
-\Big[\frac{a\lambda}{dk}p-\frac{\lambda^2}{d^2}\beta e^{-a\tau}\Big]v(t)
-\frac{b q \lambda }{cd} w(t)\\
=&-e^{-a\tau}\Big[x(t)-\frac{\lambda}{d}\Big]^2
\big[ d+\beta v(t) \big] -\frac{ap\lambda}{dk}(1-\mathcal{R}_0)v(t)
-\frac{b q \lambda}{cd} w(t)\\
\leq& \ 0,
\end{split}
\end{equation*}
and the equality holds for $x= \frac{\lambda}{d}$, $v=w=0$. Thus,
by LaSalle's invariance principle \cite{LaSalle1976}, we conclude
that $E_0$ is globally asymptotically stable.
\end{proof}

\section{Stability of the single-infection equilibrium $E_s$}

From the analysis given in the previous section, we know that
at the critical point $\mathcal{R}_0=1$, the disease-free equilibrium
$E_0$ becomes unstable and bifurcates into the single-infection equilibrium
$E_s$, which exists for $\mathcal{R}_0>1$. Thus, in order to study
the stability of $E_s$, we assume $\mathcal{R}_0>1$ in this section.
Similarly, for the local stability of $E_s$, we
have the following result.

\begin{theorem} \label{Thm3}
If $1<\mathcal{R}_0<R_1$, the single-infection equilibrium $E_s$ is
asymptotically stable; when $\mathcal{R}_0>R_1$, $E_s$ becomes
unstable.
\end{theorem}

\begin{proof}
The linearized system of model (\ref{b3}) at $E_s=(x_s,y_s,0,v_s,0)$ is
\begin{equation*}
\begin{array}{cll}
\dot x(t)&\!\!\!=\!\!\!& -(d+\beta v_s)x(t)-\beta x_sv(t),\\[1mm]
\dot y(t)&\!\!\!=\!\!\!& \beta e^{-a\tau} \big[ v_s x(t-\tau)
+x_sv(t-\tau) \big] -ay(t)-\alpha y_s w(t),\\[1mm]
\dot z(t)&\!\!\!=\!\!\!& \alpha y_sw(t)-bz(t),\\[0.5mm]
\dot v(t)&\!\!\!=\!\!\!& ky(t)-pv(t),\\[0.5mm]
\dot w(t)&\!\!\!=\!\!\!& cz(t)-qw(t),
\end{array}
\end{equation*}
with the corresponding characteristic equation given by $D_1(\xi)D_2(\xi)=0$,
where
\begin{equation*}
\begin{split}
D_1(\xi)=& \ \xi^2+(b+q)\xi+bq-\frac{c\alpha(k\beta\lambda
e^{-a\tau}-adp)}{\beta ak},\\
D_2(\xi)=&\ \xi^3+ \Big(a+p+\frac{k\beta\lambda}{ap}e^{-a\tau} \Big)\xi^2
+\Big[\frac{k\beta\lambda}{ap}e^{-a\tau}(a+p)+ap\Big]\xi\\
&\quad +k\beta\lambda e^{-a\tau}-ap(\xi+d)e^{-\xi\tau}.
\end{split}
\end{equation*}
First, note that $D_1(\xi)$ can be rewritten as
$$
D_1(\xi) = \xi^2+(b+q)\xi+bq(1-R_d),
$$
which indicates that $D_1(\xi)=0$ has two roots with negative real part
if and only if $R_d<1$ (i.e. $\mathcal{R}_0<R_1$), or one positive root
and one negative if $R_d>1$ (i.e. $\mathcal{R}_0>R_1$). Therefore,
if $\mathcal{R}_0>R_1$, the single-infection equilibrium $E_s$ is unstable.

For $D_2(\xi)=0$, we rewrite it as
\begin{equation}\label{b10}
\begin{array}{lll}
\xi^3+a_2(\tau)\xi^2+a_1(\tau)\xi+a_0(\tau)-(c_1\xi+c_2)e^{-\xi\tau}=0,
\end{array}
\end{equation}
where
\begin{equation*}
\begin{split}
a_2(\tau)&= a+p+\frac{k\beta\lambda}{ap}e^{-a\tau}, \quad
a_1(\tau)= \frac{k\beta\lambda}{ap}e^{-a\tau}(a+p)+ap, \\
a_0(\tau)&= k\beta\lambda e^{-a\tau},\quad c_1 = ap, \quad c_2 = apd.
\end{split}
\end{equation*}
It is easy to see that $\xi=0$ is not a root of (\ref{b10})
if $\mathcal{R}_0>1$, since
\begin{equation*}
a_0(\tau)-c_2= k\beta\lambda e^{-a\tau}-apd= apd(\mathcal{R}_0-1)>0.
\end{equation*}
When $\tau=0$, (\ref{b10}) becomes
\begin{equation}\label{b11}
\xi^3+a_2(0)\xi^2+(a_1(0)-c_1)\xi+a_0(0)-c_2=0.
\end{equation}
Applying the Routh-Hurwitz criterion (see \cite{Gantmacher1959}),
we know that all the roots of (\ref{b11}) have negative real part, because
\begin{equation*}
\begin{split}
&a_2(0) = a+p+\frac{k\beta\lambda}{ap}>0,\\
&a_1(0)- c_1 =  \frac{k\beta\lambda}{ap}(a+p)>0,\\
&a_0(0)-c_2 = k\beta\lambda-apd = apd(\mathcal{R}_0|_{\tau=0}-1)>0,
\end{split}
\end{equation*}
and
\begin{equation*}
\begin{split}
a_2(0)(a_1(0) \!-\! c_1) - (a_0(0) \!-\! c_2)
=& \ \Big(a+p+\frac{k\beta\lambda}{ap}\Big)
\frac{k\beta\lambda}{ap}(a+p)-(k\beta\lambda-apd)\\
=& \ \frac{k^2\beta^2\lambda^2}{a^2p^2}(a+p)
+\frac{k\beta\lambda}{ap}(a^2+ap+p^2)+apd>0.
\end{split}
\end{equation*}
Therefore, any root of (\ref{b10}) has negative real part when $\tau=0$.
As discussed in Section 3, we know that all the roots of equation (\ref{b10})
depend continuously on $\tau$. Also, (\ref{b8}) holds for (\ref{b10}),
and hence ${\rm Re}(\xi) < +\infty$ if $D_2(\xi)=0$.
Then, the roots of equation
(\ref{b10}) can only enter into the right half plane by crossing the
imaginary axis when $\tau$ increases. Thus, we define $\xi=i\varpi$
($\varpi>0$) to be a purely imaginary root of (\ref{b10}), and then obtain
\begin{equation*}
-i\varpi^3-a_2(\tau)\varpi^2+ia_1(\tau)\varpi+a_0(\tau)-(ic_1\varpi+c_2)e^{-i\varpi\tau}=0,
\end{equation*}
Taking moduli of the above equation results in
\begin{equation}\label{b12}
\begin{split}
H_s(\varpi^2)= \varpi^6&+\big[a_2^2(\tau)-2a_1(\tau)\big] \varpi^4\\
& +\big[a_1^2(\tau)-2a_0(\tau)a_2(\tau)-c_1^2\big] \varpi^2+a_0^2(\tau)-c_2^2=0.
\end{split}
\end{equation}
Since
\begin{equation*}
\begin{array}{lll}
a_2^2(\tau)-2a_1(\tau)=
a^2+p^2+d^2\mathcal{R}_0^2>0,\\ [1.5mm]
a_1^2(\tau)-2a_0(\tau)a_2(\tau)-c_1^2 = d^2(a^2+p^2)
\mathcal{R}_0^2>0,\\ [1.5mm]
a_0(\tau)^2-c_2^2 =a^2p^2d^2(\mathcal{R}_0^2-1)>0,
\end{array}
\end{equation*}
all the coefficients of $H_s(\varpi^2)$ are positive.
Then the function $H_s(\varpi^2)$ is monotonically increasing
for $0\leq \varpi^2<\infty$ with $H_s(0)>0$. This implies that equation
(\ref{b12}) has no positive roots if $\mathcal{R}_0>1$.
Hence, all the roots of (\ref{b10}) have negative real part
for $\tau>0$ if $\mathcal{R}_0>1$.
\end{proof}

Also, we we can show the global stability of $E_s$, as given in the following
theorem.

\begin{theorem} \label{Thm4}
If $1<\mathcal{R}_0<R_1$, the single-infection equilibrium $E_s$ is
globally asymptotically stable, implying that the recombinant virus
can not survive but the pathogen virus can.
\end{theorem}
\begin{proof}
We construct the Lyapunov function $V_s=V_1+\beta x_sv_s e^{-a\tau}V_2$ with
\begin{equation*}
\begin{split}
V_1&=e^{-a\tau}(x-x_s\ln{x})+(y-y_s\ln{y})+z+\frac{a}{k}(v-v_s\ln v)
+\frac{b}{c}w,\\
V_2&=\int_{t-\tau}^{t}\Big(\frac{x(\eta)v(\eta)}{x_sv_s}
-\ln\frac{x(\eta)v(\eta)}{x_sv_s}\Big)d\eta.
\end{split}
\end{equation*}

Substituting $E_s$ into (\ref{b3}) yields three identities
$G_i\equiv0$, $i=1,2,3$, where $G_1 = \lambda-dx_s-\beta x_sv_s$,
$G_2 = \beta e^{-a\tau}x_sv_s-ay_s$, $G_3 = ky_s-pv_s$.
Then we have
\begin{equation*}
\begin{split}
V_{1,x}\dot x&=V_{1,x}\dot x- e^{-a\tau} \Big(1-\frac{x_s}{x}\Big)G_1
-\frac{v}{v_s}G_2-\frac{av}{kv_s}G_3\\
&=d x_s e^{-a\tau} \Big(2-\frac{x_s}{x}-\frac{x}{x_s} \Big)
+\beta x_sv_se^{-a\tau} \Big(1-\frac{x_s}{x}-\frac{xv}{x_sv_s} \Big)
+\frac{ap}{k} v,\\
V_{1,y}\dot y&=V_{1,y}\dot y +G_2\\
& =\beta x_sv_se^{-a\tau}\Big[1+\frac{(y-y_s)x(t-\tau)v(t-\tau)}{yx_sv_s}\Big]
-ay+\alpha(y_s-y)w,\\
V_{1,v}\dot v&=V_{1,v}\dot v + \Big(1-\frac{yv_s}{y_sv} \Big)G_2
+\frac{a}{k}G_3=\beta x_sv_se^{-a\tau} \Big(1-\frac{yv_s}{y_sv}\Big)
+ay-\frac{ap}{k} v,\\
V_{1,z}\dot z&=\alpha yw-bz, \qquad V_{1,w}\dot w=bz-\frac{bq}{c}w,
\end{split}
\end{equation*}
which yields
\begin{equation}\label{b13}
\begin{split}
& \ V_{1,x}\dot x+V_{1,y}\dot y+V_{1,z}\dot z + V_{1,v}\dot v + V_{1,w}\dot w\\
=& \ \beta x_sv_se^{-a\tau} \Big[ 3-\frac{x_s}{x}-\frac{yv_s}{y_sv}
+\frac{(y-y_s)x(t-\tau)v(t-\tau)}{yx_sv_s}-\frac{xv}{x_sv_s} \Big] \\
&+d x_s e^{-a\tau} \Big(2-\frac{x_s}{x}-\frac{x}{x_s} \Big)
+\frac{\alpha dp}{\beta k}(\mathcal{R}_0-R_1)w,
\end{split}
\end{equation}
where $y_s=\frac{dp}{\beta k}(\mathcal{R}_0-1)$ has been used.
And for $V_2$, we have
\begin{equation}\label{b14}
\frac{dV_2}{dt}=\frac{xv}{x_sv_s}-\frac{x(t-\tau)v(t-\tau)}{x_sv_s}+\ln\frac{x(t-\tau)v(t-\tau)}{xv}.
\end{equation}
Combining (\ref{b13}) and (\ref{b14}) yields
\begin{equation*}
\begin{split}
\frac{dV_s}{dt}\Big|_{(3)}&=V_{1,x}\dot x+V_{1,y}\dot y+V_{1,z}\dot z
+ V_{1,v}\dot v + V_{1,w}\dot w+\beta x_sv_s e^{-a\tau}\frac{dV_2}{dt}\\
&=d x_s e^{-a\tau}\Big(2-\frac{x_s}{x}-\frac{x}{x_s} \Big)
+\frac{\alpha dp}{\beta k}(\mathcal{R}_0-R_1)w+\beta x_sv_se^{-a\tau}W,
\end{split}
\end{equation*}
where
$$W=3-\frac{x_s}{x}-\frac{yv_s}{y_sv}-\frac{y_sx(t-\tau)v(t-\tau)}{yx_sv_s}+\ln\frac{x(t-\tau)v(t-\tau)}{xv}\le0,$$
because the following inequality
\begin{equation*}
n-\sum_{i=1}^{n}\frac{b_i}{a_i}+\ln\prod_{i=1}^n\frac{b_i}{a_i}\le0,
\end{equation*}
holds for any positive $a_i$ and $b_i$ (see \cite{Kajiwara2012}). Therefore,
$\frac{dV_s}{dt}|_{(3)}\le0$ when $\mathcal{R}_0<R_1$, and the equality
holds when $x=x_s$, $y=y_s$, $v=v_s$, $w=0$.
Then, by LaSalle's invariance principle \cite{LaSalle1976},
we conclude that $E_s$ is globally asymptotically stable.
\end{proof}

\section{Stability of the double-infection equilibrium $E_d$}
At the critical point $\mathcal{R}_0= R_1$, the single-infection
equilibrium $E_s$ becomes unstable and the double-infection equilibrium
$E_d$ comes into existence for $\mathcal{R}_0> R_1$.
To discuss the stability of $E_d$, we assume $\mathcal{R}_0>R_1$ in this
section. We have the following result for the stability of $E_d$.

\begin{theorem}\label{Thm5}
For model $(\ref{b3})$, there exists an $R_2 > R_1$ such that the
double-infection equilibrium $E_d$ is asymptotically stable
for $R_1<\mathcal{R}_0<R_2$.
\end{theorem}

\begin{proof}
The linearized system of (\ref{b3}) at
$E_d=(x_d,y_d,z_d,v_d,w_d)$ is
\begin{equation}\label{b15}
\begin{array}{cll}
\dot x(t)&\!\!\!=\!\!\!& -(d+\beta v_d)x(t)-\beta x_dv(t) ,\\[1mm]
\dot y(t)&\!\!\!=\!\!\!& \beta e^{-a\tau}\big[ v_dx(t-\tau)+x_dv(t-\tau) \big]
-(a+\alpha w_d)y(t)-\alpha y_dw(t),\\[1mm]
\dot z(t)&\!\!\!=\!\!\!& \alpha w_dy(t)-bz(t)+\alpha y_dw(t),\\[1mm]
\dot v(t)&\!\!\!=\!\!\!& ky(t)-pv(t),\\[1mm]
\dot w(t)&\!\!\!=\!\!\!& cz(t)-qw(t).
\end{array}
\end{equation}
By straightforward but tedious algebraic manipulations,
we obtain the characteristic equation of (\ref{b15}), given by
\begin{equation}\label{b16}
\begin{split}
D(\xi)=&\ (\xi+p)(\xi+dR_1)\Big[\xi (\xi+b+q)
\Big(\xi+a\frac{\mathcal{R}_0}{R_1} \Big)
+abq \Big(\frac{\mathcal{R}_0}{R_1}-1 \Big)\Big] \\
&   -ap\frac{\mathcal{R}_0}{R_1}\xi(\xi+d)(\xi+b+q) e^{-\xi\tau}\\
=&\ \xi^5+\sum_{i=0}^4A_i\xi^i-\sum_{i=1}^3B_i\xi^ie^{-\xi\tau}=0,
\end{split}
\end{equation}
where
\begin{equation*}
\begin{split}
A_4 &=dR_1+a\frac{\mathcal{R}_0}{R_1}+b+p+q,\\[-1.0ex]
A_3 &=(b+p+q) \Big(dR_1+a\frac{\mathcal{R}_0}{R_1} \Big)
+p(b+q)+ad\mathcal{R}_0,\\[-1.0ex]
A_2 &= ad(b+p+q)\mathcal{R}_0
+p(b+q) \Big( d R_1 + a \frac{\mathcal{R}_0}{R_1} \Big)
+ abq \Big( \frac{\mathcal{R}_0}{R_1} -1 \Big),\\[-1.0ex]
A_1 &= adp (b+q)\mathcal{R}_0
+abq ( p + d R_1) \Big( \frac{\mathcal{R}_0}{R_1} -1 \Big),\\[-1.0ex]
A_0 &=abdpq(\mathcal{R}_0-R_1),\\[-0.5ex]
B_3 &=ap\frac{\mathcal{R}_0}{R_1},\quad
B_2=ap(b+d+q)\frac{\mathcal{R}_0}{R_1},\quad
B_1=apd(b+q)\frac{\mathcal{R}_0}{R_1},
\end{split}
\end{equation*}
showing that all $A_i \ (i=1,2,3,4)$ and $B_j \ (j=1,2,3)$
are positive for $\mathcal{R}_0 > R_1$.

When $\tau=0$, it has been shown in \cite{Jiang2009} that there exists
a constant $R_2^\ast>R_1$ such that $E_d$ is locally asymptotically
stable when $\mathcal{R}_0\in(R_1,R_2^\ast)$, implying that all the
roots of $(\ref{b16})|_{\tau=0}$ have negative real part.

Obviously, $D(\xi)$ satisfies (\ref{b8}), which implies that
$D(\xi)=0$ has no roots in the infinity ${\rm Re}(\xi)=+\infty$.
Following the procedure as shown in Sections 3 and 4,
we let $R(\varpi)$ and $S(\varpi)$ respectively be the real and imaginary
part of $D(i\varpi)\ (\varpi>0)$, given by
\begin{equation*}
\begin{split}
R(\varpi)&=A_4\varpi^4-A_2\varpi^2+A_0+B_2\varpi^2\cos(\varpi\tau)
+(B_3\varpi^2-B_1)\varpi\sin(\varpi\tau),\\
S(\varpi)&=\varpi^5-A_3\varpi^3+A_1\varpi-B_2\varpi^2\sin(\varpi\tau)
+(B_3\varpi^2-B_1)\varpi\cos(\varpi\tau).
\end{split}
\end{equation*}
Solving the equations $R(\varpi)= 0$ and $S(\varpi)= 0$
for $\sin(\varpi\tau)$ and $\cos(\varpi\tau)$, and then substituting
the results into the identity, $\sin^2(\varpi\tau)+\cos^2(\varpi\tau)=1$,
yields
$$
\frac{H(\varpi^2)}{(B_3\varpi^2-B_1)^2\varpi^2+B_2^2\varpi^4}=0
\quad \Longleftrightarrow \quad H(\varpi^2) = 0,
$$
where
$H(\varpi^2) = \varpi^{10}+a_1\varpi^8+a_2\varpi^6+a_3\varpi^4
+a_4\varpi^2+a_5$,
with
\begin{equation}\label{b17}
\begin{split}
a_1 &= A_4^2-2A_3,\\
a_2 &= 2A_1-2A_2A_4+A_3^2-B_3^2,\\
a_3 &= 2A_0A_4-2A_1A_3+A_2^2+2B_1B_3-B_2^2,\\
a_4 &= A_1^2-2A_0A_2-B_1^2,\\
a_5 &= A_0^2.
\end{split}
\end{equation}

In what follows, we shall prove that there exists an $R_2>R_1$ such that
all the roots of $H(x)=0$ have negative real part
when $\mathcal{R}_0\in(R_1,R_2)$, that is, there are no positive real
roots for $H(\varpi^2)=0$. Therefore, for $\mathcal{R}_0\in(R_1,R_2)$,
the roots of (\ref{b16}) stay in the left half complex plane and
$E_d$ is locally asymptotically stable.

The necessary and sufficient conditions for ${\rm Re}(x)<0$ when $H(x)=0$
are given by
\begin{equation*}
\begin{split}
\Delta_1 &= a_1>0,\\
\Delta_2 &= a_1a_2-a_3>0,\\
\Delta_3 &= a_3\Delta_2-a_1(a_1a_4-a_5)>0,\\
\Delta_4 &= a_4\Delta_3-a_5\big[a_2\Delta_2-(a_1a_4-a_5)\big]>0,\\
\Delta_5 &= a_5\Delta_4>0.
\end{split}
\end{equation*}
A straightforward calculation shows that
$$
\Delta_1 =  R_1^2d^2+a^2\frac{R_0^2}{R_1^2}+p^2+(b+q)^2>0,
$$
for any positive parameter values.
Obviously, $\Delta_5=A_0^2\Delta_4>0$ when $\Delta_4>0$.

For $\Delta_2$, $\Delta_3$ and $\Delta_4$, it is not easy to determine
their signs for general $\mathcal{R}_0$. Hence, we
take a continuity argument below. At $\mathcal{R}_0=R_1$,
\begin{equation*}
\begin{split}
\Delta_2|_{\mathcal{R}_0=R_1} = & \ d^4(R_1^2-1)^2F_1
+d^2(F_1^2+2d^2F_1-a^2p^2)(R_1^2-1)\\
&+\big[(b+q)^2+d^2\big](a^2+d^2+p^2)F_1>0,\\
\Delta_3|_{\mathcal{R}_0=R_1} = & \ d^2\big[d^2(a^2+p^2)R_1^4
+(a^4+a^2p^2+p^4)R_1^2+a^2p^2\big]F_2>0,
\end{split}
\end{equation*}
where
\begin{equation*}
\begin{split}
F_1 &= a^2+p^2+(b+q)^2>0,\\
F_2 &= d^2(F_1-a^2)(F_1-p^2)(R_1^2-1)+(b+q)^2\big[d^2+(b+q)^2\big]F_1>0.
\end{split}
\end{equation*}
and $\Delta_4|_{\mathcal{R}_0=R_1}=a^2d^2p^2(b+q)^2(R_1^2-1)
\Delta_3|_{\mathcal{R}_0=R_1}>0$. We know that $\Delta_i$, $i=2,\ 3,\ 4$,
continuously depend on $\mathcal{R}_0$. Hence, there exists
an $R_2\le R_2^\ast$ such that $\Delta_i$, $i=2,\ 3,\ 4$,
are all greater than zero if $R_1<\mathcal{R}_0<R_2$.
\end{proof}

In \cite{Jiang2009}, it is also proved that $E_d$ could lose its stability
through Hopf bifurcation when $\mathcal{R}_0|_{\tau=0}$ is far greater
than $R_1$. So when $\tau>0$, Hopf bifurcation may occur from $E_d$
if $\mathcal{R}_0$ is further increased from $R_1$. To obtain the critical
point at which a Hopf bifurcation takes place, we need solve the equations
$R(\varpi)= 0$ and $ S(\varpi)= 0$ for $\tau$ and $\varpi$, if we take
$\tau$ as our bifurcation parameter. Then, we can determine the
corresponding value(s) of $\mathcal R_0$, and choose the smallest
one $R_h$ satisfying $R_h>R_1$. Denote by $\tau_h$ and $\varpi_h$
the corresponding values of $\tau$ and $\varpi$.

Following \cite{Hassard1981}, there are three additional conditions which need
to be satisfied,
\begin{equation}\label{b18}
R(\varpi)\!=\! 0 \ \Rightarrow \ S(\varpi) \! \ne \! 0
\ \ \big( {\rm or} \ S(\varpi)\!=\! 0 \ \Rightarrow \ R(\varpi)
\! \ne \! 0 \big) \ \ {\rm for}\quad R_1 \!<\! \mathcal R_0 \! < \! R_h,
\end{equation}
\begin{equation}\label{b19}
\left. \frac{\partial D(\xi,\tau)}{\partial \xi}
\right|_{\xi=i\varpi_h,\tau=\tau_h}
\neq0,
\end{equation}
and
\begin{equation}\label{b20}
\left. {\rm Re}\Big(\frac{d \xi}{d \tau }\Big)
\right|_{\xi=i\varpi_h,\tau=\tau_h} < 0.
\end{equation}
The condition (\ref{b18}) implies that there are no solutions
satisfying $R(\varpi)=S(\varpi)=0\,$ if $\,\mathcal R_0 \in
(R_1, R_h)$, for which the characteristic equation $D(\xi)=0$ given
in (\ref{b16}) does not have purely imaginary roots.
From the proof of Theorem~\ref{Thm5}, we know that all roots of
$D(\xi)=0$ have negative real part for $\mathcal R_0 \in (R_1, R_h)$,
which means that the equilibrium $E_d$ is asymptotically stable
if $R_1> \mathcal R_0 < R_h$.
If all the three conditions (\ref{b18}), (\ref{b19}) and (\ref{b20}) hold,
we then conclude that (\ref{b16}) has a pair of purely imaginary roots
and all other roots with negative real part
at $\tau=\tau_h$ (i.e., at $\mathcal{R}_0=R_h$), implying existence of
a Hopf bifurcation. Therefore, at the critical point
$\tau=\tau_h$, $E_d$ loses its stability and bifurcates into
a family of limit cycles.

\section{Numerical Simulation}
In this section, we present a numerical example and some simulations by using dde23 from the software MATLAB R2012a,
to illustrate the theoretical results obtained in previous sections.

\label{table1}
\vspace{0.1in}
\begin{center}
{Table 1. Parameter notations and the sources for their values
\vspace{0.12in}
}
\begin{tabular}{llll}
\hline
 & Definition & Value(day$^{-1}$) & Source\\
\hline
$\lambda$ & Production rate of host cell       &$0\sim 10$ cell/mm$^3$
&\cite{Nelson2000}\\
$d$ & Death rate of host cell                  & 0.01
& \cite{Michie1992}\\
$\beta$ & Infection rate of host cell by virus & 0.004 mm$^3$/vir
& \cite{Revilla2003}\\
$a$ & Death rate of HIV-1 infected cell        & $0.5$
& \cite{Nelson2000}\\
$\alpha$ & Infection rate by recombinant       & Assumed $\,\alpha=\beta$
& \cite{Revilla2003}\\
$b$ &Death rate of double-infected cell        &2
& \cite{Revilla2003}\\
$k$ &HIV-1 production rate by a cell           &50 vir/cell
& \cite{Revilla2003}\\
$p$ &Removal rate of HIV-1                     &3
& \cite{Nelson2000}\\
$c$ &Production rate of recombinant            &2000 vir/cell
& \cite{Revilla2003}\\
&by a double-infected cell&&\\
$q$ &Removal rate of recombinant               &Assumed $\,q=p$
& \cite{Revilla2003}\\
\hline
\end{tabular}
\vspace{0.20in}
\end{center}

\begin{figure}[h]
\vspace{-0.3in}
\begin{center}
\includegraphics[width=1.10\textwidth]{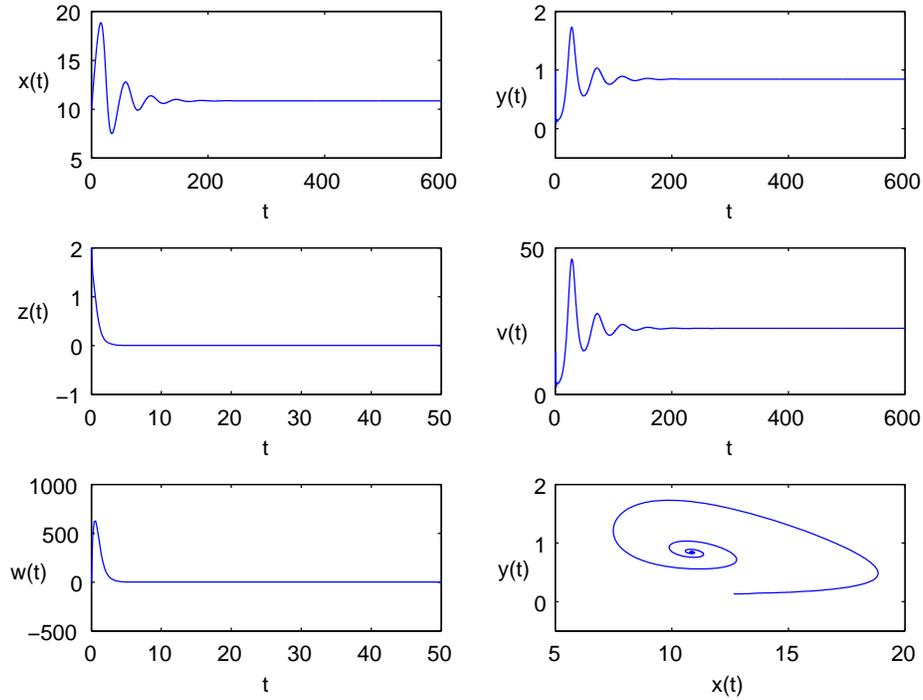}

\vspace{-0.2in}
\caption{Simulation of system (\ref{b3}) for $\tau=1.6 \in (\tau_2, \tau_1)$,
showing convergence to the stable equilibrium $E_s$.}
\label{fig1}
\end{center}
\end{figure}

The notations and typical values of the parameters used in model (\ref{b3})
are given in Table~1.
The precise value of $\tau$ is not obtained.
But it is estimated that the value of $\tau$ is between $1.0 \sim 1.5$ days
\cite{Nelson2000}. Here, we choose $\tau$ as the bifurcation parameter.

For computer simulation, we set $\lambda=1$, $d=1/180$, $\alpha=\beta=1/260$,
$a=0.5$, $b=2$, $p=q=3$, $k=80$, $c=1800$. Then,
$\mathcal R_0=480/13e^{-0.5\tau}$ and $R_1=17$.
The disease-free equilibrium $E_0$ is now given by
$$
E_0 = (180,\ 0,\ 0,\ 0,\ 0),
$$
which is globally asymptotically stable for $\tau>\tau_1=7.2176734929$,
i.e., $\mathcal R_0<1$. When $\tau<\tau_1$, $E_0$ becomes unstable and the
single-infection equilibrium $E_s$ occurs, given by
$$
E_s = (\frac{39}{8}e^{0.5\tau},\ 2e^{-0.5\tau}-\frac{13}{240},\ 0,\
\frac{160}{3}e^{-0.5\tau}-\frac{13}{9},\ 0),
$$
which is globally asymptotically stable for $\tau_1>\tau>\tau_2
=1.5512468048$. See Figure~\ref{fig1} for the simulations of system
(\ref{b3}) when $\tau=1.6$.

Further decreasing $\tau$ to pass through the critical value $\tau_2$
will cause $E_s$ to lose its stability, giving rise to the double-infection
equilibrium,
$$
E_d = (\frac{180}{17},\ \frac{13}{15},\ \frac{8}{17}e^{-0.5\tau}
-\frac{13}{60},\ \frac{208}{9},\ \frac{4800}{17}e^{-0.5\tau}-130).
$$
The corresponding characteristic equation (\ref{b16})
at the above $E_d$ becomes
\begin{equation}\label{f1}
\begin{split}
D(\xi)=& \ \xi^5+ \Big(\frac{240}{221}e^{-0.5\tau}+\frac{1457}{180} \Big)\xi^4
+(\frac{5828}{663}e^{-0.5\tau}+\frac{709}{45})\xi^3\\
&+\Big(\frac{15664}{663}e^{-0.5\tau}-\frac{19}{12}\Big)\xi^2
+\Big(\frac{4796}{221}e^{-0.5\tau}-\frac{557}{60}\Big)\xi\\
&+\frac{24}{13}e^{-0.5\tau}-\frac{17}{20}
-\Big(\frac{720}{221}\xi^3+\frac{212}{13}\xi^2+\frac{20}{221}\xi \Big)
e^{-(\xi+0.5)\tau}=0.
\end{split}
\end{equation}
Let $R(\varpi,\tau)$ and $\varpi S(\varpi,\tau)$ be the real and imaginary
parts of $D(i\varpi) \ (\varpi>0)$, yielding
\begin{equation*}
\begin{split}
R(\varpi,\tau)=& \ \Big(\frac{240}{221}e^{-0.5\tau}+\frac{1457}{180}
\Big)\varpi^4- \Big(\frac{15664}{663}e^{-0.5\tau}-\frac{19}{12}\Big)\varpi^2
+\frac{24}{13}e^{-0.5\tau}-\frac{17}{20}\\
&+ \Big(\frac{720}{221}\varpi^3-\frac{20}{221}\varpi \Big)e^{-0.5\tau}
\sin(\varpi\tau)+\frac{212}{13}\varpi^2e^{-0.5\tau}\cos(\varpi\tau),\\
S(\varpi,\tau)=& \ \varpi^4- \Big(\frac{5828}{663}e^{-0.5\tau}+\frac{709}{45}
\Big)\varpi^2+\frac{4796}{221}e^{-0.5\tau}-\frac{557}{60}\\
&+ \Big(\frac{720}{221}\varpi^2-\frac{20}{221} \Big)e^{-0.5\tau}
\cos(\varpi\tau)-\frac{212}{13}\varpi e^{-0.5\tau}\sin(\varpi\tau),
\end{split}
\end{equation*}

\begin{figure}[b]
\vspace{-0.00in}
\begin{center}
\includegraphics[width=0.7\textwidth,height=0.3\textheight 
]{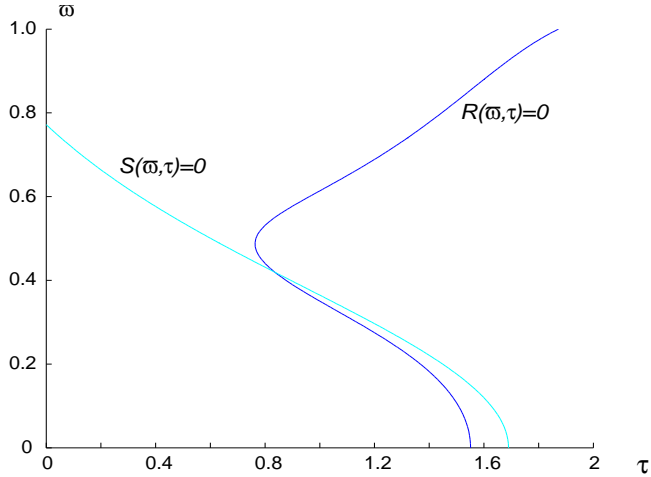}

\vspace{0.0in}
\caption{Plots of the curves $R(\varpi,\tau)=0$ and $S(\varpi,\tau)=0$
in the $\tau$-$\varpi$ plane with $(\varpi,\tau)\in[0,2.1]\times[0,2]$.}
\label{fig2}
\end{center}
\end{figure}

Solving the equations $R(\varpi,\tau)=0$ and $S(\varpi,\tau)=0$ by
using the built-in command ``fsolve'' in Maple results in
$$
(\tau_3,\varpi_3) = (0.8357983104,\ 0.4193565828).
$$
Taking into account
\begin{equation*}
\Big(\frac{720}{221}\varpi^3-\frac{20}{221}\varpi \Big)\sin(\varpi\tau)
+\frac{212}{13} \varpi^2 \cos(\varpi\tau)
\ge -\frac{4}{221}\sqrt{\varpi^2(32400\varpi^2+1)(\varpi^2+25)},
\end{equation*}
we have $R(\varpi,\tau)\ge \widetilde R(\varpi)$, where
\begin{equation*}
\begin{split}
\widetilde R(\varpi)=& \ \frac{4}{221} \Big[60\varpi^4
-\frac{3916}{3}\varpi^2+102- \varpi
\sqrt{(32400\varpi^2+1)(\varpi^2+25)}\Big]e^{-0.5\tau}\\
&+\frac{1457}{180}\varpi^4+\frac{19}{12}\varpi^2-\frac{17}{20}.
\end{split}
\end{equation*}

\begin{figure}[h]
\vspace{-0.30in}
\begin{center}
\includegraphics[width=1.18\textwidth,height=0.55\textheight]{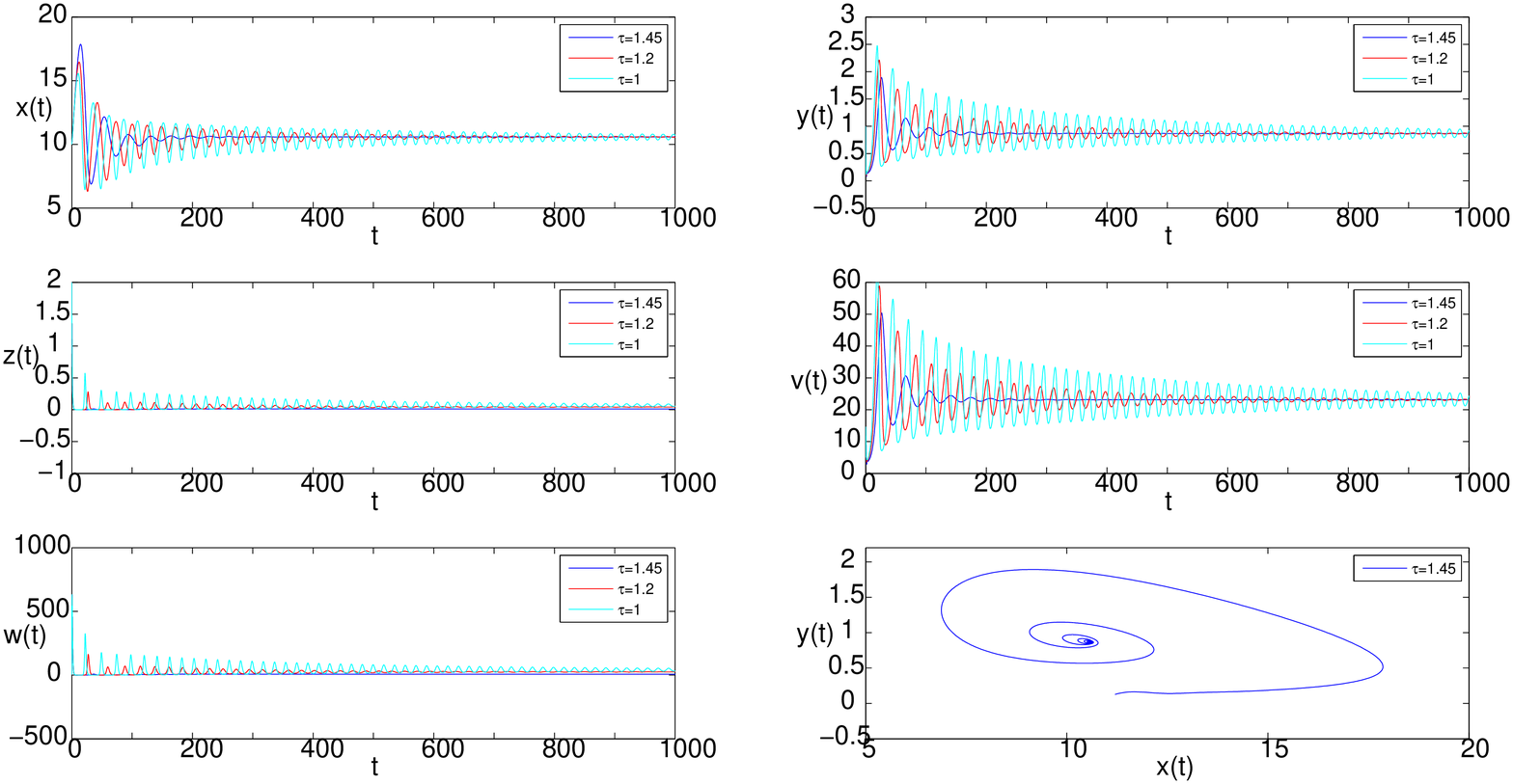}

\vspace{-0.18in}
\caption{Simulation of system (\ref{b3}) for $\tau=1.45,\ 1.2$ and $1.0$,
taken from the interval $\tau \in (\tau_3, \tau_2)$, showing convergence to
the stable equilibrium $E_d$.}
\label{fig3}
\end{center}
\vspace{-0.00in}
\end{figure}

It can be shown that for any $\tau>0$, $\widetilde R(\varpi)>0$
if $\varpi>2.1$. Thus, there are no roots of $R(\varpi,\tau)=0$ for
$\varpi>2.1$, implying that the curve $R(\varpi,\tau)=0$ in Figure~\ref{fig2}
must be below the horizontal line $\varpi=2.1$
(not shown in Figure~\ref{fig2}),
and so $(\tau_3,\varpi_3)$ is the only
intersection point. Given that all the roots of (\ref{f1})
continuously depend on $\tau$, it follows from Theorem~\ref{Thm5}
that $E_d$ is asymptotically stable when $\tau_2>\tau>\tau_3$.
The simulations for $\tau=1.45,\ 1.2$ and $1.0$ are shown
in Figure~\ref{fig3}, from which we observe that all the
components of a solution have more oscillating behaviors with larger
amplitude, and they take longer time to converge to $E_d$ when
$\tau$ is decreased from $\tau_2$ to $\tau_3$.

\begin{figure}[t]
\vspace{-0.2in}
\begin{center}
\includegraphics[width=1.10\textwidth]{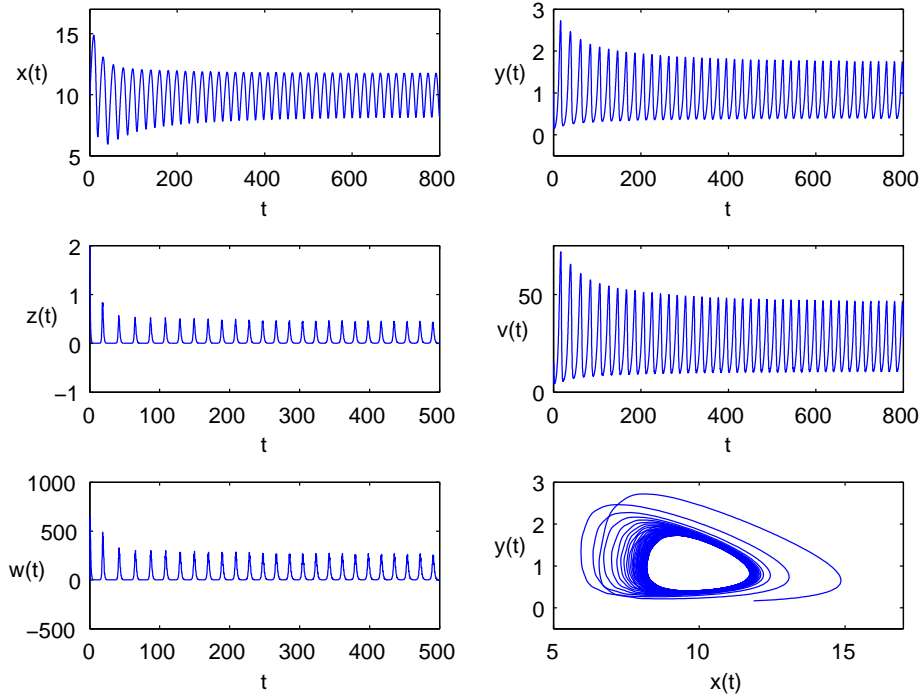}

\vspace{-0.2in}
\caption{Simulation of system (\ref{b3}) for $\tau=0.8<\tau_3$, showing
bifurcation to a stable limit cycle.}
\label{fig4}
\end{center}
\end{figure}

Finally, to consider possible Hopf bifurcation, first it is easy to see
from Figure~\ref{fig2} that
$$
S(\varpi,\tau)= 0 \quad \Longrightarrow \quad R(\varpi,\tau) < 0, \quad
{\rm for} \ \ \tau_2 < \tau < \tau_3,
$$
indicating that condition (\ref{b18}) is satisfied.
Moreover, the other two conditions also hold:
$$
\left. \frac{\partial D(\xi,\tau)}{\partial \xi}
\right|_{\xi=i\varpi_3,\tau=\tau_3}
= -9.8115344435+0.7314225159\,i \neq0,
$$
and
$$
\left. {\rm Re} \Big( \frac{d \xi}{d \tau } \Big)
\right|_{\xi=i\varpi_3,\tau=\tau_3}
= -0.0137073586 < 0.
$$
Thus, the roots of (\ref{f1}) have positive real part when $\tau<\tau_3$,
and (\ref{f1}) has a pair of purely imaginary roots at $\tau=\tau_3$,
implying existence of a Hopf bifurcation. Therefore, we conclude that
when $\tau_2>\tau>\tau_3$, the equilibrium solution $E_d$ is asymptotically
stable. At the critical point, $\tau=\tau_3$, $E_d$ loses its stability
through a Hopf bifurcation, giving rise to limit cycles.
See the simulation shown in Figure \ref{fig4}.
Further, the stability of limit cycles and the direction of bifurcations
can be determined by using the center manifold theory and normal
form theory for delay differential equations (e.g., see~\cite{Yuetla2014}).
Detailed discussions on this part are out of the scope of this paper.

\begin{figure}[h]
\vspace{-0.2in}
\begin{center}
\includegraphics[width=1.10\textwidth]{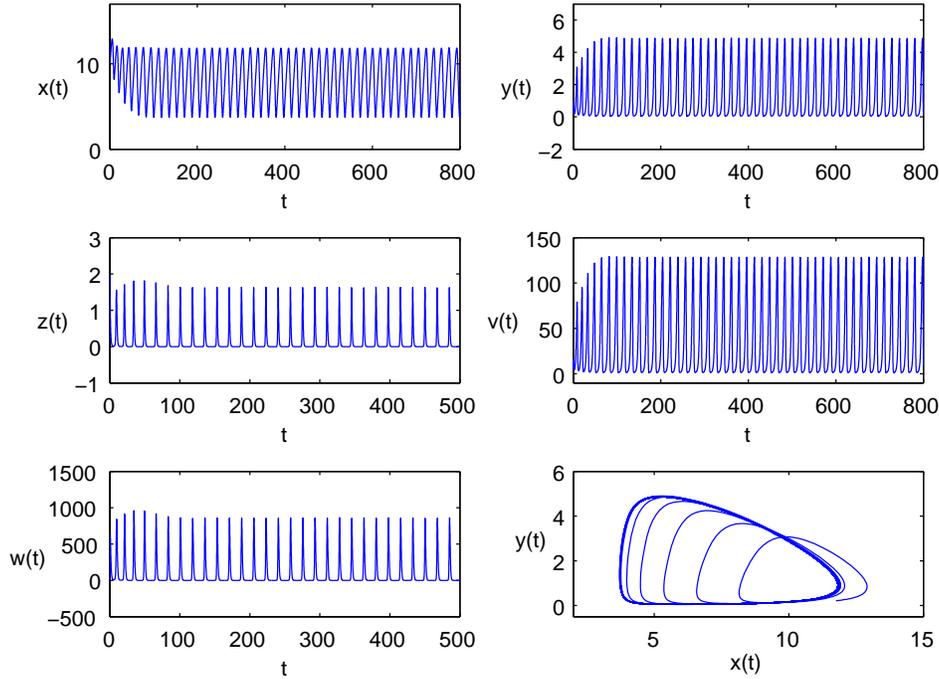}

\vspace{-0.2in}
\caption{Simulation of system (\ref{b3}) for $\tau=0$, showing
oscillating behaviour.}
\label{fig5}
\end{center}
\end{figure}

In order to demonstrate the importance of the delay to be included in
the model, in the following we will compare the results obtained above
to that given at $\tau =0$. It is easy to see that
$\left. \mathcal R_0 \right|_{\tau=0} = \frac{480}{13} > R_1 = 17$, and thus
both the disease-free equilibrium, $E_0$, and the single-infection
equilibrium, $E_1$, are unstable when $\tau=0$. To find the stability of the
double-infection equilibrium, $E_d$, we set $\tau=0$ in (\ref{f1}) to obtain
$$
D(\xi) = \xi^5 + \frac{365197}{39780} \xi^4 + \frac{211709}{9945} \xi^3
         + \frac{15209}{2652} \xi^2 + \frac{163463}{13260} \xi
         + \frac{259}{260},
$$
which yields a purely pair and three negative eigenvalues:
$0.03214833 + 0.76348925\,i$,
$-0.08306245$, $-3.91260798$, and $-5.24904353$, indicating that
$E_d$ is also unstable. Therefore, at $\tau=0$, the system must exhibit
oscillating behaviour, as shown in Figure~\ref{fig5}.
Comparing the results in this figure with that in Figure~\ref{fig4}
shows that at $\tau=0$, the solution trajectory converges much fast
to reach its steady-state value than that in Figure~\ref{fig4} for
$\tau=0.8$ More importantly, it is noted that the amplitudes of the
oscillations in Figure~\ref{fig5} is almost double of that
in Figure~\ref{fig4} though their frequencies are almost not changed.
The above observation shows that lack of even small delay in model
(\ref{b2}) can cause significant quantitative changes in solutions.
Moreover, for normal values of delay, the model
(\ref{b3}) with delay can exhibit qualitatively different behaviour,
compared with the model (\ref{b2}) without delay.
For example, at $\tau=1.2$ days, which is
within the normal range of delays $\tau \in (1.0, 1.5)$
days~\cite{Nelson2000}, model (\ref{b3}) shows convergence to
the stable double-infection equilibrium $E_d$, see Figure~\ref{fig3}.
At the marginal normal value $\tau = 1.6$, model (\ref{b3}) gives
the stable single-infection equilibrium $E_s$, see
Figure~\ref{fig1}. These significant qualitative changes due to
existence of delay can not be observed from the model (\ref{b2}) without
delay involved. This indeed suggests that the delay is a very important
fact which should not be missed in model (\ref{b2}).

\section{Conclusion and discussion}

In this paper, we present a more realistic HIV-1 model of fighting
a virus with another virus by adding delay to the model. The detailed
analytic study has shown that the improved model with delay, like the
model without delay, also has three equilibrium solutions:
the disease-free equilibrium $E_0$, single-infection equilibrium $E_s$,
and double-infection equilibrium $E_d$, and a series of bifurcations occur
as the basic reproduction number, $\mathcal{R}_0$, is increased.
It has shown that $E_0$ is globally asymptotically stable for
$\mathcal{R}_0 \in (0,1)$, and becomes unstable at the
transcritical bifurcation point $\mathcal{R}_0 =1$,
and bifurcates into $E_s$, which is globally asymptotically stable
for $\mathcal{R}_0 \in (1, R_1)$. $E_s$ loses its stability at the
another transcritical bifurcation point $\mathcal{R}_0 =R_1$,
and asymptotically stable for $\mathcal{R}_0 \in (R_1,R_h)$. Finally,
$E_d$ becomes unstable at the Hopf critical point $\mathcal{R}_0 =R_h$,
and bifurcates into a family of limit cycles.

When the delay is chosen
as the bifurcation parameter, it is shown that the delay plays an
important role in determining the dynamic behaviour of the system.
In the normal range of values, the delay may change the dynamic behaviour
quantitatively, such as greatly reducing the amplitudes of oscillations,
or even qualitatively changes the dynamical behaviour such as revoking
oscillating solutions to equilibrium solutions.
This indeed suggests that the delay is a very important
fact which should not be missed in HIV-1 modelling.

In this paper, only Hopf bifurcation has been considered. It is interesting
to know whether the model can exhibit double Hopf bifurcation if, besides
the delay, one more system parameter is chosen as second bifurcation
parameter. Another interesting question arises if we include another fact
of delay to model (\ref{b3}), that is, the existence of
virus production period for new virions to be produced within and released
from the infected cells (see~\cite{Mittler1998}).
When this second delay is included, model (\ref{b3}) becomes
\begin{equation}\label{con-1}
\begin{array}{cll}
\dot{x}(t)&\!\!\!=\!\!\!& \lambda-dx(t)-\beta x(t)v(t),\\[0.5mm]
\dot{y}(t)&\!\!\!=\!\!\!& \beta e^{-a\tau_1}
x(t-\tau_1)v(t-\tau_1)-ay(t)-\alpha w(t)y(t), \\[0.5mm]
\dot{z}(t)&\!\!\!=\!\!\!& \alpha w(t)y(t)-bz(t),\\[0.5mm]
\dot{v}(t)&\!\!\!=\!\!\!& ke^{-\tilde{a}\tau_2}y(t-\tau_2)-pv(t),\\[0.5mm]
\dot{w}(t)&\!\!\!=\!\!\!& cz(t)-qw(t),
\end{array}
\end{equation}
where $\tau_1$ and $\tau_2$ represent the latent period and
virus production period, respectively.
Then for this model, future work includes the study on the dynamical
behaviour and bifurcation patterns of the model, and how the two
delays influence stability and bifurcations. More interestingly,
with these two delays as bifurcation parameters, can the model
exhibit double Hopf bifurcation? Studying these questions
will help to well understand the impact of delays on dynamical
behaviour of HIV-1 model.

\section*{Acknowledgment}
This work was supported by the Natural Science
and Engineering Research Council of Canada (NSERC).



\medskip
Received xxxx 20xx; revised xxxx 20xx.
\medskip

\end{document}